\documentclass[12pt]{amsart}
\usepackage[colorlinks=true, pdfstartview=FitV, linkcolor=blue, citecolor=blue, urlcolor=blue]{hyperref}

\usepackage{amssymb,amsmath, amscd}
\usepackage{times, verbatim}
\usepackage{graphicx}
\usepackage[english]{babel}
 \usepackage[usenames, dvipsnames]{color}
\usepackage{amsmath,amssymb,amsfonts}
\usepackage{enumerate}
\usepackage{csquotes}
\usepackage{tabularx}
\usepackage{caption}

\usepackage{tikz}
\usetikzlibrary{calc, arrows, decorations.markings} \tikzset{>=latex}
\usepackage{amsmath,amssymb,amsfonts}
\usepackage{amscd, amssymb, latexsym, amsmath, amscd}
\usepackage[all]{xy}
\usepackage{pb-diagram}
\usepackage{enumerate}
\usepackage{ulem}
\usepackage{anysize}
\marginsize{3cm}{3cm}{3cm}{3cm}
\input xy
\xyoption{all}
\usepackage{pb-diagram}
\usepackage[all]{xy}
\usepackage{array}
\usepackage{upgreek}

\usepackage{xcolor}
\input xy
\xyoption{all}
\usepackage{ytableau}
\usepackage{mathtools}
\usepackage[overload]{empheq}

\theoremstyle{plain}
\newtheorem{theorem}{Theorem}[section]
\newtheorem{lemma}{Lemma}

\newtheorem{corollary}{Corollary}

\newtheorem{proposition}{Proposition}

\theoremstyle{definition} \theoremstyle{definition}
\newtheorem{remark}{Remark}
\newtheorem{example}{Example}

\theoremstyle{remark}

\newcommand{\G}{\textsc{\G}}

\newcommand{\Z}{\mathbb{Z}}

\newcommand{\C}{\mathbb{C}}

\def\G{{\rm G}}

\def\Sp{{\rm Sp}}

\def\GL{{\rm GL}}

\def\SO{{\rm SO}}

\def\Sym{{\rm Sym}}

\newcommand{\sla}{S_{\lambar}}

\newcommand{\lambar}{\underline{\lambda}}
\newcommand{\mubar}{\underline{\mu}}
\newcommand{\nubar}{\underline{\nu}}

\newcommand{\lamr}{{\lambda}_{1} \geq \lambda_{2} \geq \cdots \geq \lambda_{r}}
\newcommand{\mur}{{\mu}_{1} \geq \mu_{2} \geq \cdots \geq \mu_{r}}
\newcommand{\nur}{{\nu}_{1} \geq \nu_{2} \geq \cdots \geq \nu_{r}}

\newcommand{\psilr}{\Psi_{\underline{\lambda}}^{ \hspace{.2mm} r}}
\newcommand{\psimr}{\Psi_{\underline{\mu}}^{\hspace{.2mm} r}}
\newcommand{\psinr}{\Psi_{\underline{\nu}}^{\hspace{.2mm} r}}

\newcommand{\llam}{l({\underline{\lambda}})}
\newcommand{\lmu}{l({\underline{\mu}})}
\newcommand{\lnu}{l({\underline{\nu}})}

\newcommand{\clmn}{C_{{\underline{\lambda}} \hspace{.5mm} {\underline{\mu}}}^{ \hspace{.4mm}{\underline{\nu}}}}

\newcommand{\psilam}{\Psi_{\lambar}}

\newcommand{\pilam}{\Pi_{\lambar}}

\newcommand{\psimu}{\Psi_{\mubar}}

\newcommand{\pimu}{\Pi_{\mubar}}

\newcommand{\pimud}{\Pi_{\mubar^{'}}}
\newcommand{\pinu}{\Pi_{\nubar}}

\newcommand{\psimud}{\Psi_{\mubar^{'}}}
\newcommand{\psinu}{\Psi_{\nubar}}

\newcommand{\symkv}{\Sym^k (V)}
\newcommand{\glv}{\GL{(V)}}

\newcommand{\lambard}{\lambar^{'}}

\newcommand{\psilamd}{\Psi_{\lambard}}

\newcommand{\pilamd}{\Pi_{\lambard}}

\newcommand{\symlor}{\Sym^{\lambda_1}(V) \otimes \cdots \otimes \Sym^{\lambda_r}(V)}

\newcommand{\symlormo}{\Sym^{\lambda_1}(V) \otimes \cdots \otimes \Sym^{\lambda_{r-1}}(V)}

\newcommand{\symmor}{\Sym^{\mu_1}(V) \otimes \cdots \otimes \Sym^{\mu_r}(V)}

\newcommand{\xibar}{\underline{\xi}}

\newcommand{\nlkm}{N_{\lambar,k}^{\mubar}}

\newcommand{\clkn}{C_{\lambar,k}^{\nubar}}

\newcommand{\ndmrn}{N_{\mubar^{'}, \hspace{.2mm} \mu_r}^{\hspace{.4mm} \nubar}}

\newcommand{\cdmrn}{C_{\mubar^{'},\hspace{.2mm} \mu_r}^{\hspace{.4mm} \nubar}}

\newcommand{\lbx}{\lambar /{\xibar}}
\newcommand{\mbx}{\mubar /{\xibar}}

\newcommand{\pilamprodr}{\Pi_{(\lambda_1)} \otimes \Pi_{(\lambda_2)} \otimes \cdots \otimes \Pi_{(\lambda_r)} }

\newcommand{\pilamprodn}{\Pi_{(\lambda_1)} \otimes \Pi_{(\lambda_2)} \otimes \cdots \otimes \Pi_{(\lambda_n)} }

\newcommand{\pilamprodrmo}{\Pi_{(\lambda_1)} \otimes \Pi_{(\lambda_2)} \otimes \cdots \otimes \Pi_{(\lambda_{r-1})} }

\newcommand{\nlmn}{N_{{\underline{\lambda}} \hspace{.5mm} {\underline{\mu}}}^{ \hspace{.4mm}{\underline{\nu}}}}

\newcommand{\nlkn}{N_{{\underline{\lambda},} \hspace{.5mm} {k}}^{ \hspace{.4mm}{\underline{\nu}}}}

\newcommand{\mubard}{\mubar^{'}}
\newcommand{\mubart}{\tilde{\mubar}}

\newcommand{\nubart}{\tilde{\nubar}}
\newcommand{\mtm}{\tilde{\mu}_m}

\newcommand{\nubard}{\nubar^{'}}
\newcommand{\nubardd}{\nubar^{''}}

\newcommand{\cdmmtm}{C_{\mubar^{'}\hspace{.2mm} \mu_m}^{\hspace{.4mm} \mubart}}

\newcommand{\cltmn}{C_{\lambar \hspace{.5mm} \mubart}^{\nubar}}

\newcommand{\cldmdn}{ C_{ \lambar \hspace{.5mm} \mubard }^{ \nubard } }

\newcommand{\nlmdn}{ N_{ \lambar \hspace{.5mm} \mubar }^{ \nubard } }

\newcommand{\nlmno}{ N_{ \lambar \hspace{.5mm} \mubar }^{ \nubar^{(0)} } }

\newcommand{\nlmtn}{ N_{ \lambar \hspace{.5mm} \mubar }^{ \nubart } }

\newcommand{\nlmddn}{ N_{ \lambar \hspace{.5mm} \mubar }^{ \nubardd } }
\newcommand{\cdnmmn}{ C_{ \nubard \hspace{.5mm} \mu_m }^{ \nubar } }

\newcommand{\nldlrl}{ N_{ \lambard, \hspace{.5mm} \lambda_r }^{ \lambar } }

\newcommand{\com}{\Pi^{\infty}_{\mubar}}
\newcommand{\con}{\Pi^{\infty}_{\nubar}}
\newcommand{\col}{\Pi^{\infty}_{\lambar}}

\newcommand{\coltd}{\Pi^{\infty}_{\lamtd}}

\newcommand{\psmm}{\Phi_{\SO{(2m-1)}}}
\newcommand{\psmt}{\Phi_{\SO{(2m-3)}}}

\newcommand{\pspmm}{\Phi_{\Sp{(2m-2)}}}
\newcommand{\pspmt}{\Phi_{\Sp{(2m-4)}}}

\newcommand{\psemm}{\Phi_{\SO{(2m-2)}}}

\newcommand{\csom}{\Pi^{m}_{\mubar}}
\newcommand{\cson}{\Pi^{m}_{\nubar}}
\newcommand{\csol}{\Pi^{m}_{\lambar}}

\newcommand{\csoll}{\Pi^{m-1}_{\lambar}}

\newcommand{\csolt}{\Pi_{\tilde{\lambar}}^{m-1}}

\newcommand{\lamtd}{\tilde{\lambar}}

\newcommand{\lamone}{\lambar^{(1)}}
\newcommand{\lamtwo}{\lambar^{(2)}}

\newcommand{\nuone}{\nubar^{(1)}}

\newcommand{\nlmnone}{ N_{ \lambar \hspace{.5mm} \mubar }^{ \nuone } }

\newcommand{\nlknone}{ N_{ \lambar, \hspace{.5mm} k }^{ \nuone } }

\newcommand{\phign}{\Phi_{G_m}}
\newcommand{\phignone}{\Phi_{G_{m-1}}}

\begin{document}

\title[Stability condition]
{On the stability of tensor product of representations of classical Groups }

\begin{abstract}
 From an irreducible representation of $\GL{(n,\C)}$ there is a natural way to construct an irreducible representations of $\GL{(n+1,\C)}$  by adding a zero at the end of the highest weight $\lambar =( \lambda_1 \geq \lambda_2 \geq \cdots \geq \lambda_n)$ with $\lambda_i \geq 0$ of the irreducible representation of $\GL{(n,\C)}$.
 The paper considers the decomposition of tensor products of irreducible representation of $\GL{(n,\C)}$ and of the corresponding irreducible representations of $\GL{(n+1,\C)}$ and proves a stability result about such tensor products. We go on to discuss similar questions for classical groups. 
  
\end{abstract}

\author{Dibyendu Biswas}

\address{Indian Institute of Technology Bombay, Powai, Mumbai-400076}

\email{dibubis@gmail.com}
\maketitle
    {\hfill \today}
    
\tableofcontents

\section{Introduction}
An important aspect of classical groups is that they lie in the nested families:
\begin{align*}
    \GL{(n,\C)} &\subseteq \GL{(n+1,\C)} \subseteq \GL{(n+2,\C)} \subseteq \cdots \\
    \Sp{(2n,\C)} &\subseteq \Sp{(2n+2,\C)} \subseteq \Sp{(2n+4,\C)} \subseteq \cdots \\
    \SO{(2n+1,\C)} &\subseteq \SO{(2n+3,\C)} \subseteq \SO{(2n+5,\C)} \subseteq \cdots \\
   \SO{(2n,\C)} &\subseteq \SO{(2n+2,\C)} \subseteq \SO{(2n+4,\C)} \subseteq \cdots .
    \end{align*}
Further, in each case, an $n$-tuple of integers  $\lambar =(\lambda_1 \geq \ldots \geq \lambda_n \geq 0)$ gives rise to an irreducible representation of highest weight $\lambar$ of the corresponding group of rank $n$. Adding a zero at the end of $\lambar$, we thus have a natural map from irreducible representations of $\GL{(n,\C)}$ to irreducible representations of $\GL{(n+1,\C)}$, and similarly for other classical groups. In fact, let's write any of the above nested sequence of groups as 
$$G_n \subseteq G_{n+1} \subseteq G_{n+2} \subseteq \cdots  .$$
One can ask how does this natural map from irreducible representations of $G_n$ to irreducible representations of $G_{n+1}$ behave for tensor products. This paper aims to study this question. We find that the decomposition of the tensor product of irreducible representations $\pilam \otimes \pimu$ of $G_r$, when $\pilam$ and $\pimu$  come from $G_n$, is independent of $r$ as soon as $r\geq 2n$, which we may then call \textit{Stable tensor product} of $\pilam$ and $\pimu$, with the understanding that the  stable tensor product of $\pilam$ and $\pimu$ is not a representation of $G_n$ but of $G_r$ for any $r \geq 2n$. Furthermore, we found to our surprise that the stable tensor product $\pilam \otimes \pimu$ of $G_n$ is independent of the classical group ($\Sp{(2n,\C)}$, $\SO{(2n+1,\C)}$, $\SO{(2n,\C)}$) chosen.
In section~\ref{sec before stable} we describe the tensor product of all the classical groups $G_{N-1}$ if $N$ is the stablility level corresponding to that tensor product which we call \enquote{just before stability results}. 

It is possible that many of the theorems about tensor product of irreducible representations of $G_n$ have an analogue for stable tensor product, such as the Littlewood-Richardson rule or saturation conjecture, a theorem due to Knutson-Tao \cite{KT} for $\GL{(n,\C)}$, although at this point we are not sure if the statements simplify or become more involved. 
In section $\ref{sec gln}$ we begin by discussing the case of $\GL{(n,\C)}$. The main tool here and in fact for all the cases studied in this work is Pieri's rule which describes the tensor product $\psilam \otimes \Sym^k{(V)}$ where $\psilam$ is the irreducible representation of $\GL{(n,\C)}=\GL{(V)}$ with highest weight $\lambar$. Pieri's rule also proves as a consequence that every irreducible representation of highest weight of  $\lambar =(\lambda_1 \geq \ldots \geq \lambda_n \geq 0)$ of $\GL{(V)}$ is a sub-representation of $\Sym^{\lambda_1}{(V)} \otimes \Sym^{\lambda_2}{(V)} \otimes \cdots \otimes \Sym^{\lambda_n}{(V)}$, where it appears with multiplicity exactly one, and the other constituents have highest weight which are \enquote{lower} than the highest weight of this. Pieri's rule together with this consequence allows us in section \ref{sec gln} to prove our theorems comparing tensor products for $\GL{(n,\C)}$ and for $\GL{(n+1,\C)}$.

Here is the main theorem of this paper for classical groups 
proving stability theorem for classical groups of the same kind as done earlier for general linear group and at the same time we  also prove that the multiplicities occurring in the tensor product are independent of which classical group we deal with. (For a sequence of integers $\lambar = \left\{ \lamr : \lambda_{i} \geq 0 \right\}$, define $l(\lambar)$, the length of $\underline{\lambda}$, to be the largest integer $s$ such that $\lambda_{s} \neq 0$.)

\begin{theorem}\label{main thm}
Let $G_n$ be one of the groups $\Sp{(2n,\C)}$, $\SO{(2n+1,\C)}$, $\SO{(2n,\C)}$. Let $\lambar=\{ \lamr \geq 0  \}$ and $\mubar=\left\{ \mur \geq 0 \right \}$ be two sequence of integers, and $\pilam^n$ and $\pimu^n$ the corresponding irreducible representations of $G_n$ for $n \geq \text{ max }\{ \llam , \lmu\}$. Write the tensor product of $\pilam^n$ and $\pimu^n$ as,
\begin{equation}
    \pilam^n \otimes \pimu^n = \sum_{\nubar} \nlmn \left (n \right) \pinu^n .
\end{equation}
Suppose $n_0=\llam + \lmu$. Then,
\begin{enumerate}
    \item $\nlmn \left (n \right)=0$ \quad if \quad $\lnu \geq n_{0}+1$,
    \vspace{.15cm}
    \item $\nlmn \left (n \right)=\nlmn \left (n+1 \right)$ \hspace{.5mm} for \hspace{.5mm} $n \geq$ $\begin{cases}  n_0 &\text{ if } G_n = \Sp{(2n,\C)}, \SO{(2n+1,\C)}; \\ n_0 +1
    &\text{ if } G_n = \SO{(2n,\C)},
    \end{cases}$ 
    \vspace{.15cm}
    \item $\nlmn \left (n \right)$ is independent of the group $G_n=\Sp{(2n,\C)}, \SO{(2n+1,\C)}$ if $n \geq n_{0}$,
     \vspace{.15cm}
    \item $\nlmn \left (n \right)$ is independent of the group $G_n$ if $n \geq n_{0}+1.$
\end{enumerate}
\end{theorem}


\section{General Linear Group} \label{sec gln}
Any sequence of integers $\lambar=\{ \lamr  \}$  defines an irreducible representation $\psilr$ of $\GL{(r,\C)}$ for all $r\geq s= \llam$ with highest weight $\lambar=\left\{ \lamr \right \}$.

Let $\lambar$ and $\mubar$ be two sequence of integers, and $r\geq 1$, an integer such that $r \geq \llam$ and $r \geq \lmu$ then it makes sense to talk about the tensor product :
$\psilr \otimes \psimr$ of the representation of $\GL{(r,\C)}$ as $r$ varies.

Our first theorem proves that the representation $\psilr \otimes \psimr$ are \enquote{independent} of $r$ if $r \geq \llam + \lmu$. Here is a more precise statement.

\begin{theorem}\label{thm gl}
Let $\lambar=\left\{ \lamr \geq 0 \right \}$, $\mubar=\left\{ \mur \geq 0 \right \}$ be two sequence of integers with $r \geq \text{ max }\{ \llam , \lmu\}$. Write 
\begin{equation*}
    \psilr \otimes \psimr = \sum_{\nubar} \clmn \left(r\right) \psinr \quad , \quad \clmn \left(r\right) \in \Z_{\geq 0}.
\end{equation*}
Then if $r \geq \llam + \lmu$,
 $$\clmn \left(r\right) = \clmn \left(r+1\right) $$
 \vspace{1mm}
 for all sequence of integers $\nubar =\left \{  \nur \right\}$.
\end{theorem}

Proof of this theorem will be a  consequence of Pieri's Rule.

\begin{lemma}{(Pieri's Rule)}\label{Pieri GL}
Let $V$ be an $r$-dimensional $\C$-vector space, $\lambar=\{ \lambda_1 \geq \lambda_2 \geq  \cdots \geq \lambda_r \geq 0  \}$ and $\mubar=\left\{ \mur \geq 0 \right \}$ be two sequence of integers. Let $\psilr$ be the irreducible highest weight module of $\GL{(r,\C)}$. Then 
$$\psilr \otimes \symkv = \bigoplus_{\mubar} \psimr , $$
each irreducible representation $\psimu$ appearing with multiplicity 1, and exactly those $\mubar$ appear which are obtained for $\lambar$ by adding $k$-boxes to the Young diagram of $\lambar$ such that no two boxes lie in any column.
\end{lemma}

\begin{remark}
For general linear group, $\Sym^k{(V)}$ is an irreducible representation of $\GL{(V)}$, and   $\Sym^k{(V)} = \Psi_{(k)}^r $ for all $r\geq \text{ dim}(V)$.
\end{remark}

\begin{corollary} If \hspace{.2mm}
$\psilr \otimes \symkv = \bigoplus_{\mubar} \psimr $, then $\lmu \leq \llam +1$.
\end{corollary}

This Corollary will not be needed for this section, but we will have occasion to use it later.

\begin{lemma}\label{vl g sub sym}
Let $\psilam$, $\lambar=\left\{ \lamr \right \}$ be the irreducible representation of $\glv$ with highest weight $\lambar$, then 
$$ \psilam \subseteq \symlor .$$
\end{lemma}

\noindent{\bf Proof.}
The proof of this lemma will be by an induction on $r$ using Pieri's Rule. Thus we assume that the lemma is true if $\llam \leq r-1$ and then we prove it for $\lambar$ with $\llam =r$.

Let $\lambard=\left \{\lambda_1 \geq \cdots \geq \lambda_{r-1}  \right \}$. By induction hypothesis $$\psilamd \subseteq \symlormo .$$
Therefore
$$\psilamd \otimes \Sym^{\lambda_r} (V) \subseteq \symlor. $$
By Pieri's Rule $\psilam \subseteq \psilamd \otimes \Sym^{\lambda_r} (V)$, hence 
$$\psilam \subseteq \symlor. \qed$$

\noindent{\bf Proof.} {(of Theorem~\ref{thm gl})}
It suffices to prove that in the decomposition,
 $$  \psilr \otimes \psimr = \sum \clmn \left(r\right) \psinr ,$$
 if $\clmn \left(r\right) \neq 0$, then $\lnu \leq \llam + \lmu$.
 By Lemma \ref{vl g sub sym},
 $$\psimr \subseteq \symmor.$$ Therefore $$\psilr \otimes \psimr \subseteq \psilr \otimes \symmor.$$ 
 Therefore by Pieri's Rule, if 
 $$\psinr \subseteq \psilr \otimes \symmor, $$ $\lnu \leq \llam + \lmu$. By Theorem~\ref{thm all} proved in the next sections, it follows that 
  $$\clmn \left(r\right) = \clmn \left(r+1\right) . \qed$$

\begin{remark}
The very last step in the proof of Theorem~\ref{thm gl} can be done by an inductive argument as we will do for the classical groups. But we have given another proof of the last step which may be of independent interest.
\end{remark}

\section{Relating tensor products for \GL{(n)} and \GL{(n+1)}}
For $\lambar = (\lambda_1 \geq  \ldots \geq \lambda_n \geq 0)$ with corresponding irreducible highest weight representation $\psilam$ of $\GL{(n,\C)}$, it character, the Schur function $$\sla \in \Z[X_1, \ldots, X_n]^{\mathfrak{S}_n}$$
and is the character of the representation $\psilam$ at the diagonal matrix
$$
 \begin{pmatrix}
   X_{1} &  &  & \\ 
     &  \ddots & \\ 
     &   & X_{n} 
 \end{pmatrix} .
$$

Under the natural homomorphism of algebras,
\begin{align*}
    p_n: \Z[X_1, \ldots, X_{n+1}]^{\mathfrak{S}_{n+1}} &\longrightarrow \Z[X_1, \ldots, X_n]^{\mathfrak{S}_n},\\
    X_i \hspace{2.5mm} &\longrightarrow X_i, \quad i \leq n, \\
    X_{n+1} &\longrightarrow 0,
\end{align*}
it follows from the Weyl character formula or more directly from the corresponding determinantal formula that 
\begin{align*}
    p_n(\sla)= \begin{cases}
    0 &\text{ if } \lambar = (\lambda_1 \geq  \ldots \geq \lambda_{n+1} \geq 0), \lambda_{n+1}\neq 0, \vspace{1.2mm}\\ \sla & \text{ if } \lambar = (\lambda_1 \geq  \ldots \geq \lambda_{n+1} \geq 0),   \lambda_{n+1}= 0.
    \end{cases}
\end{align*}

\begin{theorem}\label{thm all}
Let $G_n=\GL{(n,\C)}$.
For $\lambar = (\lambda_1 \geq  \ldots \geq \lambda_{n+1} \geq 0)$, $\mubar = (\mu_1 \geq  \ldots \geq \mu_{n+1} \geq 0)$, let $\psilam$, $\psimu$ be the corresponding irreducible highest weight representations of $G_{n+1}$. Suppose $\llam \leq n$, $\lmu \leq n$, and assume that 
\begin{equation}\label{all eqn}
    \psilam \otimes \psimu =\sum  \clmn \psinu
\end{equation}
is the decomposition of the tensor products of $\psilam$ and $\psimu$ as irreducible representation of $G_{n+1}$. Assume that if $\clmn \neq 0$, then $\lnu \leq n$. Then the decomposition in equation (\ref{all eqn}) of the tensor product of $\psilam$ and $\psimu$, treated as representation of $G_{n+1}$ is the same as if $\psilam$ and $\psimu$ were treated as representation of $G_n$.
\end{theorem}

\begin{proof}
The proof of this theorem amounts to the fact that $p_n$ is a
homomorphism of algebras and take character of irreducible representation of $G_{n+1}$ to those of $G_n$ under the hypothesis that the last entry of the highest weight of the representation of $G_{n+1}$ is zero. More precisely, the result for the tensor product for $\GL{(n,\C)}$ is obtained from the result for the tensor product for $\GL{(n+1,\C)}$  by simply deleting those representations of $\GL{(n+1,\C)}$ for which $l(\nubar)=n+1$; in particular these are at least as many irreducible representation in $\psilam \otimes \psimu$ on $\GL{(n+1,\C)}$ as on $\GL{(n,\C)}$.
\end{proof}

\begin{example}
We give decomposition of $\pi_{\lambar} \otimes \pi_{\mubar}$ for $\lambar=(2,1,1,0,0,\ldots)$ and  $\mubar=(1,1,0,0,0,\ldots)$ on appropriate general linear groups. All these calculation were done on \textit{Lie Software}. 
\end{example}

 \begingroup
\setlength{\tabcolsep}{20pt} 
\renewcommand{\arraystretch}{2} 
\begin{tabular} 
{ |m{1.9em}| m{4.7cm}| m{5.2cm} | }
 \hline 
  Groups & \centering $\Psi_{\lambar} \otimes \Psi_{\mubar}$ &  $\Psi_{\lambar} \otimes \Psi_{\mubar}$ as sum of irreducible representations \\
 \hline
 $\GL{(5)}$ & \centering (2,1,1,0,0) $\otimes$ (1,1,0,0,0) &  \setlength{\baselineskip}{17pt}  (2,2,2,0,0) + (2,2,1,1,0) + (2,1,1,1,1) + (3,2,1,0,0)\quad +\quad (3,1,1,1,0) \\
 \hline
 $\GL{(6)}$ & \centering (2,1,1,0,0,0) $\otimes$ (1,1,0,0,0,0)  &  \setlength{\baselineskip}{17pt} (2,2,2,0,0,0) + (2,2,1,1,0,0) + (2,1,1,1,1,0) + (3,2,1,0,0,0)\quad+ (3,1,1,1,0,0) \\
\hline
\end{tabular}
\captionof{table}{Decomposition for \GL{(n)}}\label{table gl}
\endgroup

\vspace{1cm}

\begin{remark}
 In Table~\ref{table gl} for $\GL{(n)}$, the tensor product stabilizes at $\GL{(5)}$ level, with $5=\llam + \lmu$, i.e., the result for the tensor product for $\GL{(6)}$ is obtained by just adding a 0 at the end of the corresponding result for $\GL{(5)}$ which verifies the Theorem~\ref{thm all}.
\end{remark}

\section{Pieri's formula for Classical Groups and Consequences}\label{sec pieri}

For $V$ a finite dimensional vector space over $\C$, let $q:V \to \C$ be a non-degenerate quadratic form on $V$, considered as an element of $\Sym^2(V^{*})$, giving rise to the contraction map $$\Sym^2(V^{*}) \times \Sym^k(V) \to \Sym^{k-2}(V).$$ Let $\Pi_{(k)}$ be the kernel of the contraction map from $\Sym^k V$ to $\Sym^{k-2} V$. Then  $\Pi_{(k)}$ is the irreducible representation of $\SO{(2n+1,\C)}$ with highest weight $k L_{1}$ and 
$$
 \Sym^k V = \Pi_{(k)}  \oplus \Pi_{(k-2)} \oplus \cdots \oplus \Pi_{(k-2p)}, 
$$
where $p$ is the largest integer $\leq k / 2$  (c.f. Section 19.5 of \cite{FH}).

For a sequence of integers $\lambar = \left\{ \lamr : \lambda_{i} \geq 0 \right\}$, let $l(\lambda)$ be the length of $\underline{\lambda}$, the largest integer $s$ such that $\lambda_{s} \neq 0$. Let $|\lambar|$ be the sum of nonzero parts of a weight $\lambar$. For any two sequence of integers $\lambar$ and $\xibar$ such that the Young diagram of $\lambar$ contains the Young diagram of $\xibar$, define $|\lambar / \xibar| = |\lambar| - |\xibar|$. We say that $\lambar / \xibar$ is a horizontal strips if $\lambda_1 \geq \xi_1 \geq \lambda_2 \geq \xi_2 \geq \cdots \geq \lambda_{l} \geq \xi_l  \geq \cdots$.




 We will use Pieri's theorem from \cite{SO}. Pieri's formula in \cite{SO} for $\Sp{(2n)}$, $\SO{(2n+1)}$ and $\SO{(2n)}$ has slightly different formulation but we will look only at those irreducible representations whose highest weights have the last coordinate equal to zero for $\Sp{(2n)}$, $\SO{(2n+1)}$ or the last two coordinates are zero for $\SO{(2n)}$. In this cases the Pieri's rule from \cite{SO} simplifies to the following theorem.
 
\begin{theorem}\label{pieri thm}
Let $G_n$ be any of the classical groups $\Sp{(2n,\C)}$, $\SO{(2n+1,\C)}$, $\SO{(2n,\C)}$. Let $\pilam$ be an irreducible highest weight representation with highest weight $\lambar = (\lambda_1 \geq  \cdots \geq \lambda_{n} \geq 0)$. Assume that if $G_n= \Sp{(2n,\C)} \text{ and } \SO{(2n+1,\C)}$, $\lambda_n=0$ and if $G_n= \SO{(2n,\C)}$, $\lambda_{n-1}=\lambda_n=0$, then 

\begin{equation*}
    \pilam \otimes \Pi_{(k)} = \bigoplus_{\mubar} \nlkm  \pimu 
\end{equation*}
with $\nlkm$ given by:
$$\nlkm = {\#} \left \{ \xibar : \text{ sequence of integers } \xibar \text{ satisfying the following two conditions}  \right\} $$

\begin{enumerate}
    \item $\lambar / \xibar$ and $\mubar / \xibar$ are both horizontal strips, \vspace{1.5mm}
    \item $|\lambar / \xibar| + |\mubar / \xibar|=k$. 
\end{enumerate}

\end{theorem}

\begin{proof}
The symplectic Pieri's formula (Lemma~\ref{Pieri Sp}) 
is the same as what is asserted above. For odd orthogonal groups the Pieri's formula (Lemma~\ref{pieri odd}) 
allows the condition  $|\lambar / \xibar| + |\mubar / \xibar|=k \text{ and }k-1$. But the condition  $|\lambar / \xibar| + |\mubar / \xibar|=k-1$ does not occur as $\llam \neq n$. 
Hence the Pieri's formula (Lemma~\ref{pieri odd}) 
simplifies to the formula asserted above. 
Similarly for the even orthogonal groups, in Lemma~\ref{pieri even}, the condition $\lambda_{n-1}=\lambda_n=0$  implies $\xi_{n-1}=\xi_n=0$, hence the Pieri's formula (Lemma~\ref{pieri even}) 
reduces to the assertion in our theorem.
\end{proof}

\begin{corollary} \label{prod decom}
Let $G_n$ be any of the classical groups $\Sp{(2n,\C)}$, $\SO{(2n+1,\C)}$, $\SO{(2n,\C)}$. Let $\pilam$ be an irreducible highest weight representation of $G_n$ with highest weight $\lambar = (\lambda_1 \geq  \cdots \geq \lambda_{n} \geq 0)$. Assume that if $G_n= \Sp{(2n,\C)} \text{ and } \SO{(2n+1,\C)}$, $\lambda_n=0$ and if $G_n= \SO{(2n,\C)}$, $\lambda_{n-1}=\lambda_n=0$. Then, 
$$ \pilam \subseteq \pilamprodn .$$
\end{corollary}

\noindent{\bf Proof.}
Assuming the conditions in the theorem, we have $\lambar = (\lambda_1 \geq  \cdots \geq \lambda_{r} \geq 0)$ where $r=n-1$ for $\Sp{(2n,\C)}$ and $\SO{(2n+1,\C)}$, and $r=n-2$ for $\SO{(2n,\C)}$. Note that $\Pi_{(0)}=\C$, so it is enough to prove that $$ \pilam \subseteq \pilamprodr. $$
We give a proof by an induction on $r$, using Pieri's Theorem~\ref{pieri thm}. Thus we assume that the Corollary is true if $\llam \leq r-1$, and then prove it for $\lambar$ with $\llam =r$.

Let $\lambard=\left \{\lambda_1 \geq \cdots \geq \lambda_{r-1}  \right \}$. By induction hypothesis, $$\pilamd \subseteq \pilamprodrmo .$$
Therefore $$\pilamd \otimes \Pi_{(\lambda_r)} \subseteq \pilamprodr. $$
Now $\nldlrl$ is a positive integer as $\lambard$ satisfies both the conditions in Pieri's Theorem~\ref{pieri thm}, so  $$\pilam \subseteq \pilamd \otimes \Pi_{(\lambda_r)}, $$
hence  $$\pilam \subseteq \pilamprodr. \qed$$

\begin{corollary}\label{lemma length}
Let $G_n$ be any of the classical groups $\Sp{(2n,\C)}$, $\SO{(2n+1,\C)}$, $\SO{(2n,\C)}$. Let $\pilam$ be an irreducible highest weight representation of $G_n$ with highest weight $\lambar = (\lambda_1 \geq  \cdots \geq \lambda_{n} \geq 0)$. Assume that if $G_n= \Sp{(2n,\C)} \text{ and } \SO{(2n+1,\C)}$, $\lambda_n=0$ and if $G_n= \SO{(2n,\C)}$, $\lambda_{n-1}=\lambda_n=0$. If $$\pilam \otimes \Pi_{(k)} = \bigoplus_{\mubar} \nlkm \left(n \right) \pimu ,$$  with $\nlkm \left(n \right) >0$, then $\lmu \leq \llam +1$. 
\end{corollary}

\begin{proof}
 Let $\llam =l$. So by assumption $l\leq n-1$ for $\Sp{(2n, \C)}$ and $\SO{(2n+1, \C)}$ and $l\leq n-2$ for $\SO{(2n, \C)}$. By Pieri's Theorem~\ref{pieri thm}, $\lambar / \xibar$ and $\mubar / \xibar$ are horizontal strips,
\begin{align*}
  &\lambda_1 \geq \xi_1 \geq \lambda_2 \geq \xi_2 \geq \cdots \geq \lambda_{l} \geq \xi_l \geq \lambda_{l+1} \geq \xi_{l+1} \geq \cdots , \\
  & \mu_1 \geq \xi_1 \geq \mu_2 \geq \xi_2 \geq \cdots \geq \mu_{l+1} \geq \xi_{l+1} \geq \mu_{l+2} \geq  \cdots .
\end{align*}
 So $\lambda_{l+1}=0$ implies  $\xi_{l+1}=0$ which further gives $\mu_{l+2}=0$. Hence $\lmu \leq \llam +  1$.
\end{proof}

For our purposes these corollaries are enough but these corollaries are true without putting $\lambda_n=0$ or $\lambda_{n-1}=\lambda_{n}=0$.  More precise versions of the corollaries will need more precise versions of the Pieri's Rule which are slightly different for different classical groups.

\vspace{1mm}
\begin{corollary} \label{prop order}
Let $G_n$ be any of the classical groups $\Sp{(2n,\C)}$, $\SO{(2n+1,\C)}$, $\SO{(2n,\C)}$. Let $\pimu$ be an irreducible highest weight representation with highest weight $\mubar=(\mu_1 \geq \mu_2 \geq \cdots \geq \mu_{n-1} \geq \mu_n \geq 0)$. Suppose, $$r = \begin{cases} n &\text{ if  } \Sp{(2n,\C)}, \SO{(2n+1,\C)},\\
n-1 &\text{ if  } \SO{(2n,\C)}.
\end{cases} $$
 Assume $\lmu=r$ and define $\mubard=(\mu_1 \geq \mu_2 \geq \cdots \geq \mu_{r-1})$, then we have the following isomorphism of  $G_n$-modules,
 $$\pimud \otimes \Pi_{\mu_r} \cong \pimu  \bigoplus_{\substack{\nu_r < \mu_r\\ \lnu \leq r}}
 \ndmrn \pinu .$$
\end{corollary}

\begin{proof}
We will appeal to Theorem~\ref{pieri thm} to prove this proposition. 
Let $\nubar =(\nu_1 \geq \nu_2 \geq \cdots \geq \nu_n \geq 0)$.  We will prove that $\nu_r = \mu_r$ implies $\nubar = \mubar$. So it is enough to prove 

$$\ndmrn = \begin{cases}
1 & \text{ when }  \nubar = \mubar \\
0 & \text{ when }  \nu_r > \mu_r,
\end{cases}
$$
where $\ndmrn$ is the cardinality of the set of $\xibar$ 
satisfying the conditions that $\mubard/\xibar$ and $\nubar/\xibar$ are both horizontal strips and $|\mubard/\xibar| + |\nubar/\xibar| =\mu_r$. Horizontal strip condition of $\mubard/\xibar$ implies $\xi_r=0$ and we get the equation
\begin{equation} \label{sp 01}
    \sum_{i=1}^{r-1}\left(\mu_{i}-\xi_{i}\right) + \sum_{i=1}^{r-1}\left(\nu_{i}-\xi_{i}\right)= \mu_r - \nu_r
\end{equation}
from the condition $|\mubard/\xibar| + |\nubar/\xibar| =\mu_r$.

If $\mu_r = \nu_r$ then RHS of the equation (\ref{sp 01}) becomes zero, and since $\mubard/\xibar$ and $\nubar/\xibar$ are horizontal strip, we have $\mu_{i} \geq \xi_{i}$ and $\nu_{i} \geq \xi_{i}$, hence $\xi_i=\mu_i$ and $\nu_{i} = \xi_{i}$ for $1 \leq i \leq r-1$.
So we get $\nu_i=\mu_i$ for $1 \leq i \leq r-1$ and  $\xibar=\mubard$ 
The former implies $\nubar = \mubar$ and the later tells $$\ndmrn=1 \text{ for } \nubar = \mubar.$$

If $\nu_r > \mu_r$, then  RHS of the equation (\ref{sp 01}) becomes negative which is not possible as LHS is always positive as $\mubard/\xibar$ and $\nubar/\xibar$ are horizontal strip. So no such $\xibar$ exists for this case and hence $\ndmrn =0$.
\end{proof}

\section{Classical Groups}\label{sec clas grp}
For the convenience of the reader we state  Theorem~\ref{main thm} from the introduction again.
\vspace{1mm}
\begin{theorem}\label{thm stability}
Let $G_n$ be one of the groups $\Sp{(2n,\C)}$, $\SO{(2n+1,\C)}$, $\SO{(2n,\C)}$. Let $\lambar=\{ \lamr \geq 0  \}$ and $\mubar=\left\{ \mur \geq 0 \right \}$ be two sequence of integers, and $\pilam^n$ and $\pimu^n$ the corresponding irreducible representations of $G_n$ for $n \geq \text{ max }\{ \llam , \lmu\}$. Write the tensor product of $\pilam^n$ and $\pimu^n$ as,
\begin{equation}
    \pilam^n \otimes \pimu^n = \sum_{\nubar} \nlmn \left (n \right) \pinu^n .
\end{equation}
Suppose $n_0=\llam + \lmu$. Then,
\vspace{1mm}
\begin{enumerate}
    \item $\nlmn \left (n \right)=0$ \quad if \quad $\lnu \geq n_{0}+1$,
    \vspace{.15cm}
    \item $\nlmn \left (n \right)=\nlmn \left (n+1 \right)$ \hspace{.5mm} for \hspace{.5mm} $n \geq$ $\begin{cases}  n_0 &\text{ if } G_n = \Sp{(2n,\C)}, \SO{(2n+1,\C)}; \\ n_0 +1
    &\text{ if } G_n = \SO{(2n,\C)},
    \end{cases}$ 
     \vspace{.15cm}
    \item $\nlmn \left (n \right)$ is independent of the group $G_n=\Sp{(2n,\C)}, \SO{(2n+1,\C)}$ if $n \geq n_{0}$,
    \vspace{.15cm}
    \item $\nlmn \left (n \right)$ is independent of the group $G_n$ if $n \geq n_{0}+1$.
\end{enumerate}
\end{theorem}

{ {\bf Proof.}}
The proof of this theorem will be by an inductive process which will prove all the three conclusions (1), (2), (3) and (4) at the same time. The induction involved will be using an ordering on the sequence of integers $\mubar=\left\{ \mur \geq 0 \right \}$ in which we declare that $\mubar > \mubart=(\tilde{\mu}_1 \geq \cdots \geq \tilde{\mu}_r \geq 0)$ if 
\begin{enumerate}
    \item either $\lmu > l(\mubart)$, or 
    \item If $\lmu = l(\mubart)=r$, then $\mu_r > \tilde{\mu}_r$.
\end{enumerate}
With this ordering on the sequence of integers $\mubar=( \mur \geq 0 )$, Corollary~\ref{prop order} asserts that 
 $$\pimud^n \otimes \Pi_{\mu_r}^n = \pimu^n  \hspace{.2cm} \bigoplus_{\substack{\nu_r < \mu_r\\ \lnu \leq r}} \hspace{1mm} \ndmrn \left (n \right) \hspace{.1cm} \pinu^n .$$
 
 Now we give the inductive argument of the proof of this theorem for $ \pilam^n \otimes \pimu^n$ assuming that the theorem is true for $ \pilam^n \otimes \Pi_{\mubart}^n$ whenever $\mubart < \mubar$.
 
 Since we are assuming that our theorem is true for $ \pilam^n \otimes \Pi_{\mubart}^n$ for $\mubart < \mubar$, in particular, it is true for $ \pilam^n \otimes \Pi_{\mubard}^n$, and hence also for $ \pilam^n \otimes \Pi_{\mubard}^n \otimes \Pi_{\mu_r}$ by Pieri's Theorem~\ref{pieri thm}. Hence our theorem is also true for 
 $$ \pilam^n \otimes \left( \pimu^n  \bigoplus_{\substack{\nu_r < \mu_r\\ \lnu \leq r}} \ndmrn \left (n \right) \pinu^n   \right) .$$
 Again by inductive hypothesis, the theorem is true for all $\nubar$ with $\nubar < \mubar$ appearing in the above sum. Therefore as a consequence, our theorem is true for the remaining term in the above sum, i.e. $ \pilam^n \otimes \pimu^n$. \qed

\section{Just Before Stability level }\label{sec before stable}

We will be following the notation of section~\ref{sec pieri}. Theorem~\ref{thm before stable} is a further improvement of the Theorem~\ref{thm stability} in the sense now one can talk about the multiplicity of the highest weight irreducible representation appearing in tensor product of all the classical groups $G_{N-1}$ if $N$ is the stablility level corresponding to that tensor product.    We also discuss about the parity of $|\nubar|$ corresponding to $\nubar$ appearing in the tensor product for some ranges.

Let us recall Pieri's rule for symplectic group and odd orthogonal group respectively.  

\begin{lemma}{(Pieri's Rule \cite{SO}) }\label{Pieri Sp}
Let  $\lambar=\left\{ \lambda_1 \geq \lambda_2 \geq \cdots \geq \lambda_n \geq 0 \right \}$ be a sequence of integers, $k$ a positive integer and $\pilam^n$ be the corresponding irreducible highest weight module of $\Sp{(2n,\C)}$. If 
\begin{equation}
    \pilam^n \otimes \Pi_{(k)}^n = \bigoplus_{\mubar} \nlkm \left (n \right) \pimu^n, 
\end{equation}
then, 
$$\nlkm (n) = {\#} \left \{ \xibar : \lbx \text{ and }  \mbx \text{ are both horizontal strips and } |\lbx| + |\mbx|=k  \right\} .$$

\end{lemma}

\vspace{1mm}

\begin{lemma}{(Pieri's Rule \cite{SO})} \label{pieri odd}
Let $\lambar$ and $\mubar$ be sequence of integers of length $\leq n$, $k$ a nonnegative integer and $\pilam^n$ be the corresponding irreducible highest weight module of $\SO{(2n+1,\C)}$. If 
\begin{equation}
    \pilam^n \otimes \Pi_{(k)}^n = \bigoplus_{\mubar} \nlkm \left (n \right) \pimu^n, 
\end{equation}
then,
$$\nlkm (n)= {\#} \left \{ \xibar : \text{ sequence of integers } \xibar \text{ satisfying the following three conditions}  \right\} .$$

\begin{enumerate}
    \item $\lambar / \xibar$ and $\mubar / \xibar$ are both horizontal strips, \vspace{1mm}
    \item $|\lambar / \xibar| + |\mubar / \xibar|=k$ or $k-1$, \vspace{1mm}
    \item If $|\lambar / \xibar| + |\mubar / \xibar|=k-1$, then $l(\xibar)=l(\lambar)=n$.
\end{enumerate}

\end{lemma}

These two Lemmas are needed for the proof of the 2nd part of the Theorem~\ref{thm before stable}.

\begin{lemma}(\cite{SO})\label{pieri even}
Let  $\lambar=\left\{ \lambda_1 \geq \lambda_2 \geq \cdots \geq \lambda_n \geq 0 \right \}$ be a sequence of integers, $k$ a positive integer and $\pilam^n$ be the corresponding irreducible highest weight module of $\SO{(2n,\C)}$. Let  $\mubar =\{\mu_{1} \geq \cdots \geq \mu_{n-1} \geq |\mu_{n}| \geq 0\}$ be a sequence of integers. If 

\begin{equation}
    \pilam^n \otimes \Pi_{(k)}^n = \bigoplus_{\mubar} \nlkm \left (n \right) \pimu^n, 
\end{equation}
then,
$$\nlkm (n)= {\#} \left \{ \xibar : \text{ sequence of integers } \xibar \text{ satisfying the following four conditions}  \right\} .$$

\begin{enumerate}
    \item $\xi_{1} \geq \xi_{2} \geq \cdots \geq \xi_{n-1} \geq\left|\xi_{n}\right|$,\vspace{1mm}
    \item  $\lambda_{1} \geq \xi_{1} \geq \lambda_{2} \geq \xi_{2} \geq \cdots \geq \xi_{n-1} \geq$ $\lambda_{n} \geq \xi_{n}$ \quad and \vspace{1mm} \\ 
    $\mu_{1} \geq \xi_{1} \geq \mu_{2} \geq \xi_{2} \geq \cdots \geq \xi_{n-1} \geq \mu_{n} \geq \xi_{n}$, \vspace{1mm}
    \item $\sum_{i=1}^{n}\left(\lambda_{i}-\xi_{i}\right) + \sum_{i=1}^{n}\left(\mu_{i}-\xi_{i}\right)=k$, \vspace{1mm}
    \item $\xi_{n} \in\left\{ \lambda_{n}, \mu_{n}\right\}$.

\end{enumerate}

\end{lemma}

Lemma~\ref{pieri even} is needed to prove 3rd part of the Theorem~\ref{thm before stable}.

\vspace{1mm}
\begin{theorem}\label{thm before stable}
Let $G_n$ be one of the groups $\Sp{(2n,\C)}$, $\SO{(2n+1,\C)}$, $\SO{(2n,\C)}$. Let $\lambar=\{ \lamr \geq 0  \}$ and $\mubar=\left\{ \mur \geq 0 \right \}$ be two sequence of integers, and $\pilam^n$ and $\pimu^n$ the corresponding irreducible representations of $G_n$ for $n \geq \text{ max }\{ \llam , \lmu\}$. Write the tensor product of $\pilam^n$ and $\pimu^n$ as,
\begin{equation}
    \pilam^n \otimes \pimu^n = \sum_{\nubar} \nlmn \left (n \right) \pinu^n .
\end{equation}
Suppose $n_0=\llam + \lmu$. Then,
\vspace{.3cm}
\begin{enumerate}
    \item \label{one prelevel} Let $\nlmn (n) \neq 0$. Then, 
    \begin{equation*}
        |\nubar| \equiv |\lambar| + |\mubar| \mod 2
    \end{equation*}
      for $\quad$
$n \geq$ $\begin{cases}  \text{ max } \{\llam, \lmu \} & \quad \text{ if } \quad G_n = \Sp{(2n,\C)}, \SO{(2n,\C)}; \\ \hspace{1cm} n_0
    & \quad \text{ if } \quad G_n = \SO{(2n+1,\C)}.\\
    \end{cases}$     
  \vspace{.3cm}  

 \item \label{two prelevel} Let $\nubar=\{\nu_1 \geq \ldots \geq \nu_{n_0-1} \geq 0\}$. Define $d_{\nubar}:=|\lambar| + |\mubar|-|\nubar|$ and $\nubar^{(1)}:=(\nu_1,\ldots,\nu_{n_0-1},1)$ if $ \nu_{n_0-1} \geq 1$. Then, 
 \begin{equation*} 
     \nlmn \left (n_0-1 \right)= \begin{cases}\nlmn \left (n_0 \right)  & \text{ if } \hspace{.2cm} G_n = \Sp{(2n,\C)}, \vspace{.25cm}\\
\begin{rcases}
\nlmn \left (n_0 \right)  & \text{ if } d_{\nubar} \equiv 0 \mod 2
     \vspace{.1cm}\\
     \nlmnone \left (n_0 \right)  & \text{ if } d_{\nubar} \equiv 1 \mod 2  \vspace{.15cm}
\end{rcases}  & \text{ if } \hspace{.2cm} G_n = \SO{(2n+1,\C)}.
       \\
    \end{cases}
 \end{equation*}
     
 \vspace{.3cm}

   \item \label{three prelevel} Let $G_n = \SO{(2n,\C)}$. Let  $\nubar=\{\nu_1 \geq \ldots \geq \nu_{n_0} \geq 0\}$ and define  $\nubar^{(0)}:=(\nu_1,\ldots,\nu_{n_0-1}, -\nu_{n_0})$ if $\nu_{n_0} \geq 1$. Then, 
   \vspace{.15cm}
   \begin{enumerate}
       \item $\nlmno \left (n_0 \right)= \nlmn \left (n_0 \right)$, \vspace{1.59mm}
       \item $ \nlmn \left (n_0 \right)= 
 \nlmn \left (n_0 +1 \right).$
   \end{enumerate}   
\end{enumerate}
\end{theorem}

\vspace{2mm}
{\bf Proof of (1)}
Parity for $\SO{(2n)}$ and $\Sp{(2n)}$ follows by looking at the action of the central character which is the action of $- \text{Id} \in \SO{(2n)} \text{ or } \Sp{(2n)}$ on any irreducible representation.
By Schur's Lemma, $- \text{Id}$ acts by $\pm 1$ on any irreducible representation of $\SO{(2n, \C)}$ or $\Sp{(2n, \C)}$, hence also on their tensor product. The conclusion on the parity of $|\nubar|$ with $\nlmn (n) \neq 0$ easily follows.

For the groups $\SO{(2n+1)}$ we don't have central character but surprisingly we have parity relation for the values of $n\geq n_0$. We prove the parity result for $\SO{(2n+1)}$ by Theorem~\ref{pieri thm} (Pieri's rule) and Corollary~(\ref{prod decom}). 

Let $\mubar$ is a positive integer $k$ then $n_0=\llam +1$. By Theorem~\ref{pieri thm}, any two highest weights $\nubar$ and $\nubard$ appearing in the tensor product satisfy the equations $|\lambar| + |\nubar| - 2 |\xibar|=k$ and
$|\lambar| + |\nubard| - 2 |\xibar^{'}|=k$ for some $\xibar$, $\xibar^{'}$ respectively.
Hence \begin{equation*}
    |\nubar| - |\nubard|=2(|\xibar|-|\xibar^{'}|).
\end{equation*}
It means any $|\nubar|$ and $|\nubar^{'}|$ have same parity. Now
$\lambar + \mubar=(\lambda_1 + \mu_1, \lambda_2 + \mu_2, \ldots, \lambda_n + \mu_n)$ always appears in the tensor product $\pilam^n \otimes \pimu^n$, see Ex.7, section 21 of \cite{JEH} .
So $|\nubar| \equiv |\lambda| + |\mubar|$ is true for $\lmu=1$ case.

For a general $\mubar$, by the Corollary~(\ref{prod decom}),
$$\pilam^n \otimes \pimu^n \subseteq \pilam^n \otimes ( \Pi_{\mu_1}^n \otimes \cdots \otimes \Pi_{\mu_n}^n ).$$
Now by repeated application of the Pieri's case (Theorem~\ref{pieri thm}), the later tensor product contains all the same parity weights $\nubar$, hence  $\pilam^n \otimes \pimu^n$ also contains same parity highest weights. As $\lambar + \mubar$ appears in tensor product so
$$|\nubar| \equiv |\lambda| + |\mubar| \mod 2 \quad  \text{ for }  \quad n\geq n_0.$$

\vspace{1mm}
{\bf Proof of (2)}
Let $\lambar =\{\lambda_1 \geq \cdots \geq \lambda_r \geq 0 \}$ be a sequence of integers and $n \geq r$. Define a set $C(\lambar,n)$ of sequences $\xibar=\{\xi_1 \geq \cdots \geq \xi_n \geq 0 \}$ such that  
$$\xibar \in C(\lambar,n) \quad \text{ if and only if } \quad \lambda_1 \geq \xi_1  \geq \cdots \geq \lambda_{n} \geq \xi_{n} \geq 0. 
$$
Observe that if $\lambda_{n+1}=0$, then 
 $$ \xibar \in C(\lambar,n)  \iff  \xibar \in  C(\lambar,n+1).$$
 
\noindent {\bf Symplectic group }: We prove
\begin{equation}\label{eq sym}
    \nlmn (n_0-1) = \nlmn (n_0) \quad \text{for} \quad l(\nubar)\leq n_0-1,
\end{equation}
 by a similar induction as  used in Theorem~\ref{thm stability}. So it is enough to prove 
 \begin{equation}\label{sym mu one}
     \nlkn (m) = \nlkn (m+1) \quad \text{for} \quad l(\nubar)\leq m,
 \end{equation}
where $m=\llam$, assuming $\mubar =k$, a positive integer in equation~(\ref{eq sym}). By Pieri's rule (Lemma~\ref{Pieri Sp}),
$$ \nlkn (m)=\# \bigl\{\xibar= \{ \xi_1 \geq  \cdots \geq \xi_{m} \geq 0 \} \text{ satisfies condition~(\ref{sym con}) below} \bigr\}$$
\begin{equation}\label{sym con}
\left\{  \begin{aligned}
         & \text{(i) } \xibar \in  C(\lambar,m), \\ \vspace{.5mm}
  & \text{(ii) } \xibar \in  C(\nubar,m), \\ \vspace{.5mm}
  & \text{(iii) } \sum_{i=1}^m (\lambda_i + \nu_i -2 \xi_i) =k. 
    \end{aligned}\right.
\end{equation}
Notice that 
\begin{align} 
    \xibar \in  C(\lambar,m) \quad & \iff \quad \xibar \in  C(\lambar,m+1) && \text{ as } \llam =m, \nonumber \\[0.7ex] 
    \xibar \in  C(\nubar,m) \quad & \iff \quad \xibar \in  C(\nubar,m+1) && \text{ as } l(\nubar)\leq m. \nonumber 
\end{align} 
Since $m=\llam$, $\xibar \in  C(\lambar,m)$ implies $\xi_{m+1}=0$. Hence we have
\begin{align*}
    \sum_{i=1}^m (\lambda_i + \nu_i -2 \xi_i) =k
     \quad & \iff \quad \sum_{i=1}^{m+1} (\lambda_i + \nu_i -2 \xi_i) =k   
\end{align*}
as $ \lambda_{m+1}=\nu_{m+1}=\xi_{m+1}=0$. So 
$$\nlkn (m) = \nlkn (m+1).$$
Hence our theorem for symplectic group  follows.
\vspace{.5cm}

\noindent{\bf Odd Orthogonal Group }: Similar induction arguments works also for odd orthogonal group, so it is enough to prove
\begin{equation*}
\nlkn(m) =
\begin{cases}
\nlkn \left (m+1 \right)  & \text{ if } d_{\nubar} \equiv 0 \mod 2
     \vspace{.1cm}\\
     \nlknone \left (m+1 \right)  & \text{ if } d_{\nubar} \equiv 1 \mod 2  \vspace{.15cm}
\end{cases}    
\end{equation*}
for $ l(\nubar)\leq m$,
where $m=\llam$. 
Define the set $N_m(\nubar)$ by
$$N_m(\nubar) =  \{\xibar =( \xi_1 \geq \cdots \geq \xi_{m} \geq 0 ) \quad | \quad \xibar \text{ satisfies condition~(\ref{odd con}) below}\}.$$
\begin{equation}\label{odd con}
\left\{
    \begin{aligned}
   & \text{(i) } \xibar \in  C(\lambar,m), \\ \vspace{.5mm}
  & \text{(ii) } \xibar \in  C(\nubar,m), \\ \vspace{.5mm}
  & \text{(iii) } \sum_{i=1}^m (\lambda_i + \nu_i -2 \xi_i) =k \text{ or } k-1, \\
  & \text{(iv) } \text{If } \sum_{i=1}^m (\lambda_i + \nu_i -2 \xi_i) = k-1, \text{ then } \llam=l(\xibar)=m.
    \end{aligned}\right.
\end{equation}
By Pieri (Lemma~\ref{pieri odd}), $\nlkn (m) = \# N_m(\nubar)$. Since $\llam = m$, both the case $k$ and $k-1$ happen in (\ref{odd con}) and we can write the set $N_m(\nubar)$ as a  disjoint union of two sets 
\begin{equation}\label{disjoint union}
    N_m(\nubar)= N_m^{'}(\nubar) \sqcup N_m^{''}(\nubar)
\end{equation}
where $N_m^{'}(\nubar)$ and $N_m^{''}(\nubar)$ are defined by as follows. 
$$N_m^{'}(\nubar) = \{\xibar =( \xi_1 \geq \cdots \geq \xi_{m} \geq 0 ) \quad | \quad \xibar \text{ satisfies condition~(\ref{odd con a}) below} \}$$
\begin{equation}\label{odd con a}
\left\{
    \begin{aligned}
        & \text{(i) } \xibar \in  C(\lambar,m) \\ \vspace{.5mm}
  & \text{(ii) } \xibar \in  C(\nubar,m) \\ \vspace{.5mm}
  & \text{(iii) } \sum_{i=1}^m (\lambda_i + \nu_i -2 \xi_i) =k, 
    \end{aligned} \right.
\end{equation}
$$  N_m^{''}(\nubar) = \{\xibar =( \xi_1 \geq \cdots \geq \xi_{m} \geq 1 ) \quad | \quad \xibar \text{ satisfies condition~(\ref{odd con b}) below} \}$$
\begin{equation}\label{odd con b}
\left\{
    \begin{aligned}
         & \text{(i) } \xibar \in  C(\lambar,m) \\ \vspace{.5mm}
  & \text{(ii) } \xibar \in  C(\nubar,m) \\ \vspace{.5mm}
  & \text{(iii) } \sum_{i=1}^m (\lambda_i + \nu_i -2 \xi_i) = k-1
    \end{aligned} \right.
\end{equation}
We have $$ N_m^{'}(\nubar)=  N_{m+1}(\nubar), $$
by the same argument that we used to prove the equation~(\ref{sym mu one}) for the symplectic group. Again 
$$  N_m^{''}(\nubar)=  N_{m+1}(\nubar^{(1)}), $$
the equality is due to both (\ref{length odd}) and (\ref{nu equal odd}) happen.
\begin{equation}\label{length odd}
    \sum_{i=1}^m (\lambda_i + \nu_i -2 \xi_i)  =k-1 \quad \iff \quad
\sum_{i=1}^{m+1} (\lambda_i + \nu^{(1)}_i -2 \xi_i)  =k
\end{equation}
\begin{equation}\label{nu equal odd}
    \{\xibar \in  C(\nubar,m), \xi_{m}\geq 1\} \quad \iff \quad \{ \xibar \in  C(\nubar^{(1)},m+1), \xi_{m+1}=0 \}
\end{equation}
(\ref{length odd}) holds because $\sum_{i=1}^{m+1} \nu^{(1)}_{i}= \sum_{i=1}^{m}\nu_{i}+1$ as $\nubar^{(1)}=(\nu_1,\ldots,\nu_{m},1)$ and $\lambda_{m+1}=\xi_{m+1}=0$ as $m=\llam$. (\ref{nu equal odd}) trivially holds. 

By Pieri's Theorem~\ref{pieri thm},
\begin{equation}\label{disjoint sum}
    \nlkn (m+1) = \# N_{m+1}(\nubar) \quad \text{and} \quad \nlknone (m+1)= \# N_{m+1}(\nubar^{(1)}). 
\end{equation}
Combining (\ref{disjoint union}) and (\ref{disjoint sum}), we have 
 $$\nlkn (m) \hspace{1mm} = \hspace{1mm} \nlkn (m+1) \hspace{1mm} +  \hspace{1mm} \nlknone (m+1).$$
   
 Using part~(1) of the Theorem~\ref{thm before stable},
 \begin{align*}
    & \nlkn (m+1) \neq 0 && \text{ if } |\nubar| \equiv |\lambar| + |\mubar| \mod 2 \\
   & \nlknone (m+1) \neq 0 && \text{ if } |\nubar^{(1)}| \equiv |\lambar| + |\mubar| \mod 2,
 \end{align*}
  hence the product
  $$\nlkn \left (n_0 \right) \cdot \nlknone \left (n_0 \right) =0 $$
as $|\nubar|$ and $|\nubar^{(1)}|$ do not have the same parity. So our theorem is true for odd orthogonal groups.

\vspace{1mm}
{\bf Proof of (3)}
As $\llam, \lmu < n_0$, it is obvious that the outer automorphism of $\SO{(2n_0)}$ given by conjugating by an element of $\mathrm{O(2n_0)}$ not in $\SO{(2n_0)}$ preserves $\pilam^{n_0}$, $\pimu^{n_0}$ hence also the tensor product $\pilam^{n_0} \otimes \pimu^{n_0}$. Therefore,
$$\nlmno \left (n_0 \right)= \nlmn \left (n_0 \right)$$
and part~(a) follows.

For part~(b) again a similar induction argument works as for the other two groups. So it is enough to prove 
$$\nlkn(m) = \nlkn(m+1) \quad \text{for} \quad l(\nubar)\leq m
.$$
Define two set $N_m(\nubar)$ and $N_{m+1}(\nubar)$ as follows.
$$ N_m(\nubar) =  \{\xibar =( \xi_1 \geq \cdots \geq \xi_{m-1} \geq |\xi_{m}| ) \quad | \quad \xibar \text{ satisfies condition~(\ref{cond even}) below} \}$$
 \begin{equation}\label{cond even}
  \left\{
      \begin{aligned}
          & \text{(i) }\lambda_{1} \geq \xi_{1} \geq \lambda_{2} \geq \xi_{2} \geq \cdots \geq \lambda_{m-1} \geq  \xi_{m-1} \geq 0 \geq \xi_{m} \\ \vspace{.5mm}
  & \text{(ii) } \nu_{1} \geq \xi_{1} \geq \nu_{2} \geq \xi_{2} \geq \cdots \geq \nu_{m-1} \geq \xi_{m-1} \geq \nu_{m} \geq \xi_{m} \\ \vspace{.5mm}
  & \text{(iii) } \sum_{i=1}^m (\lambda_i + \nu_i -2 \xi_i) =k \\
  & \text{(iv) } \xi_{m} \in\left\{ 0, \nu_{m}\right\}. 
      \end{aligned}\right.
  \end{equation}
$$ N_{m+1}(\nubar) =  \{\xibar =( \xi_1 \geq \cdots \geq \xi_{m} \geq |\xi_{m+1}| ) \quad | \quad \xibar \text{ satisfies condition~(\ref{cond even b}) below} \}$$
  \begin{equation}\label{cond even b}
  \left\{
      \begin{aligned}
           & \text{(i) }\lambda_{1} \geq \xi_{1} \geq \lambda_{2} \geq \xi_{2} \geq \cdots \geq \lambda_{m-1} \geq  \xi_{m-1} \geq 0 \geq \xi_{m} \geq 0 \geq \xi_{m+1} \\ \vspace{.5mm}
  & \text{(ii) } \nu_{1} \geq \xi_{1} \geq \nu_{2} \geq \xi_{2} \geq \cdots \geq \nu_{m-1} \geq \xi_{m-1} \geq \nu_{m} \geq \xi_{m}\geq 0 \geq \xi_{m+1} \\ \vspace{.5mm}
  & \text{(iii) } \sum_{i=1}^{m+1} (\lambda_i + \nu_i -2 \xi_i) =k \\
  & \text{(iv) } \xi_{m+1}=0. 
      \end{aligned} \right.
  \end{equation}
By Lemma~\ref{pieri even},
\begin{align*}
    \nlkn(m)= \# N_m(\nubar) \quad \text{ and } \quad \nlkn(m+1)= \# N_{m+1}(\nubar). 
\end{align*}
So it is now enough to show
$$N_m(\nubar) = N_{m+1}(\nubar).$$
If $\xi_m=0$, (\ref{cond even}) and (\ref{cond even b}) are identical. If $\xi_m < 0$, then $\nu_m=\xi_m <0$ by (iv) of (\ref{cond even}), which contradicts our assumption $\nu_m\geq 0$. So $\xi_m=0$. Therefore, $$ N_m(\nubar) = N_{m+1}(\nubar).$$
Hence our thoerem follows.\qed

\section{Examples}
In this section, we give decomposition of $\pi_{\lambar} \otimes \pi_{\mubar}$ for $\lambar=(2,1,1,0,0,\ldots)$, $\mubar=(1,1,0,0,0,\ldots)$ on appropriate classical groups. All these calculation were done on \textit{Lie Software} to verify various part of the Theorem~\ref{thm before stable}.

\vspace{.5cm}
\begingroup
\setlength{\tabcolsep}{20pt} 
\renewcommand{\arraystretch}{2} 
\begin{tabular}{ | m{3em} | m{5cm}| m{5.5cm} | }
 \hline
 Groups & \centering $\Pi_{\lambar} \otimes \Pi_{\mubar}$ &  $\Pi_{\lambar} \otimes \Pi_{\mubar}$ as sum of irreducible representations \\
 \hline
 $\Sp{(6)}$ & (2,1,1) $\otimes$ (1,1,0) & \setlength{\baselineskip}{17pt} (2,2,2) + (1,1,0) + (2,2,0) + (2,1,1) + (3,2,1) + (2,0,0) + (3,1,0) \\ \hline
 $\Sp{(8)}$  & (2,1,1,0) $\otimes$  (1,1,0,0) & \setlength{\baselineskip}{17pt}
 (1,1,1,1) + (2,2,2,0) + (1,1,0,0) + (2,2,1,1) + (2,2,0,0)  + (3,2,1,0) + (2,0,0,0) + (3,1,1,1) + (3,1,0,0) + 2$\times$(2,1,1,0)\\
\hline
 $\Sp{(10)}$ & (2,1,1,0,0) $\otimes$ (1,1,0,0,0) & \setlength{\baselineskip}{17pt}
(1,1,1,1,0) + (2,2,2,0,0) + (1,1,0,0,0) + (2,2,1,1,0) + (2,2,0,0,0) + (2,1,1,1,1)  + (3,2,1,0,0) + (2,0,0,0,0) + (3,1,1,1,0) + (3,1,0,0,0) + 2$\times$(2,1,1,0,0)\\
 \hline
\end{tabular}
\captionof{table}{Decomposition for \Sp{(2n)}}\label{table sp}
\endgroup

\vspace{1cm}

\begin{remark}
We make the following observations for parity relation from these tables.
\begin{enumerate}
    \item Here $|\lambar| + |\nubar|=4+2=6$. $|\nubar|$ remains always even for all $\nubar$ appearing in the tensor products for $\Sp{(6)}$, $\Sp{(8)}$ and $\Sp{(10)}$ in Table~\ref{table sp}. Similarly it holds in Table~\ref{table even}  for $\SO(2n)$. 
    
  \item
  $|\nubar|$ is even for all $\nubar$ appearing in the tensor product for $\SO{(11)}$ but it is not true for $\SO{(9)}$, in Table~\ref{table odd}. As an example the weights $(2,2,0,0)$ and $(2,1,1,1)$ appearing in tensor product for $\SO{(9)}$ have different parity.  
\end{enumerate}
This verifies Part~(\ref{one prelevel}) of Theorem~\ref{thm before stable}.
\end{remark}

\vspace{.5mm}
\begingroup
\setlength{\tabcolsep}{20pt} 
\renewcommand{\arraystretch}{2} 
\begin{tabular}{ | m{2em} | m{5.3cm}| m{5.8cm} | }
 \hline
 Groups & \centering $\Pi_{\lambar} \otimes \Pi_{\mubar}$ &  $\Pi_{\lambar} \otimes \Pi_{\mubar}$ as sum of irreducible representations \\
 \hline
 $\SO{(10)}$ & \centering (2,1,1,0,0) $\otimes$ (1,1,0,0,0) & \setlength{\baselineskip}{19pt}
(1,1,1,1,0) + (2,2,2,0,0) + (1,1,0,0,0) + (2,2,1,1,0) + (2,2,0,0,0) + (2,1,1,1,1)   + (3,2,1,0,0) + (2,0,0,0,0) + (3,1,1,1,0) + (3,1,0,0,0)  + (2,1,1,1,-1) \quad + \quad  2$\times$(2,1,1,0,0)\\
 \hline
 $\SO{(12)}$ & \centering (2,1,1,0,0,0) $\otimes$ (1,1,0,0,0,0) & \setlength{\baselineskip}{19pt} (1,1,1,1,0,0) + (2,2,2,0,0,0) + (1,1,0,0,0,0) + (2,2,1,1,0,0) + (2,2,0,0,0,0) + (2,1,1,1,1,0) + (3,2,1,0,0,0) + (2,0,0,0,0,0) + (3,1,1,1,0,0) + 
 (3,1,0,0,0,0) \hspace{.95mm}+ 2$\times$(2,1,1,0,0,0) \\
\hline
\end{tabular}
\captionof{table}{Decomposition for \SO{(2n)}}\label{table even}
\endgroup
\vspace{1cm} 
 

\begingroup
\setlength{\tabcolsep}{20pt} 
\renewcommand{\arraystretch}{2} 
\begin{tabular}{ | m{3em} | m{5cm}| m{5cm} | }
 \hline
 Groups & \centering $\Pi_{\lambar} \otimes \Pi_{\mubar}$ &  $\Pi_{\lambar} \otimes \Pi_{\mubar}$ as sum of irreducible representations \\
 \hline
 $\SO{(9)}$  & (2,1,1,0) $\otimes$  (1,1,0,0) & \setlength{\baselineskip}{17pt}
 (1,1,1,1) + (2,2,2,0) + (1,1,0,0) + (2,2,1,1) + (2,2,0,0) + (2,1,1,1) + (3,2,1,0) + (2,0,0,0) + (3,1,1,1) + (3,1,0,0) + 2$\times$(2,1,1,0) \\
\hline
 $\SO{(11)}$ & (2,1,1,0,0) $\otimes$ (1,1,0,0,0) & \setlength{\baselineskip}{17pt}
(1,1,1,1,0) + (2,2,2,0,0) + (1,1,0,0,0) + (2,2,1,1,0) + (2,2,0,0,0) + (2,1,1,1,1) + (3,2,1,0,0) + (2,0,0,0,0) + (3,1,1,1,0) + (3,1,0,0,0) + 2$\times$(2,1,1,0,0) \\
 \hline
\end{tabular}
\captionof{table}{Decomposition for \SO{(2n+1)}}\label{table odd} 
\endgroup

\vspace{1cm}

\begin{remark} The following observations verify Part~(\ref{two prelevel}) and (\ref{three prelevel}) of the Theorem~\ref{thm before stable}  for the classical groups when we take $\lambar=(2,1,1,0,0,\ldots)$ and $\mubar=(1,1,0,0,0,\ldots)$.
\begin{enumerate}
     \item Comparing the tensor products of $\Sp{(8)}$ and $\Sp{(10)}$, one can check $$\nlmn(4)=\nlmn(5)$$
      holds for all $\nubar$ with $l(\nubar)\leq 4$ in the Table~\ref{table sp}.
     
     \item All $|\nubar|$ are even except for the weight $(2,1,1,1)$ in the tensor product decomposition of $\SO{(9)}$ in Table~\ref{table odd}. Comparing the tensor product decomposition of $\SO{(9)}$ and $\SO{(11)}$, 
     $$\nlmn(4)=\nlmn(5) \text{ for all }  \nubar, \hspace{.5mm} \text{ except }  \nubar=(2,1,1,1)$$
     where $l(\nubar)\leq 4$. For the weight  $\nubar=(2,1,1,1)$ , we have
     $$N_{\lambar \hspace{.5mm} \mubar}^{(2,1,1,1)}(4)=N_{\lambar \hspace{.5mm} \mubar}^{(2,1,1,1,1)}(5)=1 .$$
     
     \item In Table~\ref{table even}, for $\nubar=(2,1,1,1,-1)$,
     $$N_{\lambar \hspace{.5mm} \mubar}^{(2,1,1,1,-1)}(5) = N_{\lambar \hspace{.5mm} \mubar}^{(2,1,1,1,1)}(5)=1 $$
     and for other $\nubar$,
     $$\nlmn(5)=\nlmn(6).$$
\end{enumerate}
\end{remark}

\vspace{1mm}
\noindent{\bf Acknowledgement:} This work is part of author's thesis written at IIT Bombay and the author thanks `IIT Bombay' for PhD fellowship. The author thanks Prof. Dipendra Prasad for suggesting this question, for numerous helpful discussions, and for spending a lot of time going through the paper and correcting mistakes. After the first version of this work was uploaded on the arXiv, I learnt in correspondence with Prof. Soichi Okada that parts of Theorem~\ref{main thm} are due to~\cite{KoTe}. We thank Prof. Okada for bringing this paper to our attention. We also thank Prof. Vinay Wagh for helping to perform the Lie software in local desktop.

\end{document}